\documentclass{amsart}
\usepackage{amsfonts}
\usepackage{amsmath}

\newtheorem{theorem}{Theorem}
\theoremstyle{plain}
\newtheorem{acknowledgement}{Acknowledgement}

\newtheorem{corollary}{Corollary}

\newtheorem{proposition}{Proposition}
\newtheorem{remark}{Remark}

\numberwithin{equation}{section}
\input{tcilatex}
\input tcilatex

\begin{document}
\title{From random walks to rough paths}
\author{Emmanuel Breuillard, Peter Friz and Martin Huesmann}
\address{Laboratoire de Mathematiques, Universite Paris 11,\\
91405 Orsay, France; DPMMS, University of Cambridge, CB3 0WB,\ UK; Math.
Institut, Universit\"{a}t Bonn, D-53115. Corresponding author:
P.K.Friz@statslab.cam.ac.uk}

\begin{abstract}
Donsker's invariance principle is shown to hold for random walks in rough
path topology. As application, we obtain Donsker-type weak limit theorems
for stochastic integrals and differential equations.
\end{abstract}

\maketitle

\section{Introduction}

Consider a random walk in $\mathbb{R}^{d}$, given by the partial sums of a
sequence of independent random-variables $\left( \xi _{i}:i=1,2,3,\dots
\right) $, identically distributed, $\xi _{i}\overset{\mathcal{D}}{=}\xi $
with zero-mean and unit variance. \textit{Donsker's theorem} (e.g. \cite%
{revuz-yor-1999}) states that the rescaled and piecewise-linearly-connected
random-walk 
\begin{equation*}
W_{t}^{\left( n\right) }=\frac{1}{n^{1/2}}\left( \xi _{1}+\dots +\xi _{\left[
tn\right] }+\left( nt-\left[ nt\right] \right) \xi _{\left[ nt\right]
+1}\right)
\end{equation*}%
converges weakly to $d$-dimensional Brownian motion $B$ in the sense that 
\begin{equation*}
\mathbb{E}\left( f\left( W^{\left( n\right) }\right) \right) \rightarrow 
\mathbb{E}\left( f\left( B\right) \right) \text{ as }n\rightarrow \infty ,
\end{equation*}%
for all continuous, bounded functionals $f$ on $C\left( \left[ 0,1\right] ,%
\mathbb{R}^{d}\right) $ with uniform topology; we shall use the shorter
notation 
\begin{equation*}
W^{\left( n\right) }\implies B\text{ on }C\left( \left[ 0,1\right] ,\mathbb{R%
}^{d}\right)
\end{equation*}%
to decribe this type of convergence. It was observed by Lamperti in \cite%
{MR0143245} that this convergence takes place in $\alpha $-H\"{o}lder
topology, i.e. 
\begin{equation*}
W^{\left( n\right) }\implies B\text{ on }C^{\alpha \text{-H\"{o}l}}\left( %
\left[ 0,1\right] ,\mathbb{R}^{d}\right)
\end{equation*}%
for $\alpha <\left( p-1\right) /2p$ provided $\mathbb{E}\left( \left\vert
\xi \right\vert ^{2p}\right) <\infty ,$ $p>1$ and the condition relating $%
\alpha $ and $p$ is sharp. In particular, for convergence in $\alpha $-H\"{o}%
lder topology for any $\alpha <1/2$ one needs finite moments of any order.

Since $t\mapsto W_{t}^{\left( n\right) }$ is a (random) Lipschitz path, it
can be canonically lifted by computing iterated integrals up to any given
order, say $N$. The resulting, lifted, path is denoted by 
\begin{equation*}
t\mapsto S_{N}(W_{t}^{\left( n\right) })\equiv \left( 1,W_{t}^{\left(
n\right) },\int_{0}^{t}W_{s}^{\left( n\right) }\otimes W_{s}^{\left(
n\right) },\dots \right) 
\end{equation*}%
and can be viewed as a (random) path with values in the step-$N$ free
nilpotent group, realized as the subset $G^{N}\left( \mathbb{R}^{d}\right) $
of the tensor-algebra $\mathbb{R\oplus R}^{d}\oplus \dots \oplus \left( 
\mathbb{R}^{d}\right) ^{\otimes N}$. See \cite{chen-57} and \cite{lyons-98,
lyons-qian-02, MR2314753} for background on iterated integrals and rough
paths. One can ask if $S_{N}\left( W_{\cdot }^{\left( n\right) }\right) $
converges weakly to a limit. The obvious candidate is Brownian motion $%
\mathbf{B}\ $on the step-$N$ free nilpotent group i.e. the symmetric
diffusion with generator given by the sub-Laplacian on the stratified Lie
group $G^{N}\left( \mathbb{R}^{d}\right) $ associated to the covariance
matrix of $\mathbf{\xi }$. In uniform topology, the answer is affirmative
and follows from work of Stroock--Varadhan \cite{MR0517406}, see also \cite%
{MR1009453}.

There is motivation from \textit{rough path theory} to work in stronger
topologies than the uniform one. Indeed, various operations of SDE theory
and stochastic integration theory are \textit{continuous }functions of 
\textit{enhanced Brownian motion} (i.e. Brownian motion plus L\'{e}vy's
area, or equivalently: Brownian motion on the step-$2$ free nilpotent group)
in rough path sense. More specifically, any $G^{2}\left( \mathbb{R}%
^{d}\right) $-valued path $\mathbf{x}\left( \cdot \right) $ has canonically
defined path increments, denoted by $\mathbf{x}_{s,t}:=\mathbf{x}_{s}^{-1}%
\mathbf{x}_{t}$, so that $\left\Vert \mathbf{x}_{s,t}\right\Vert =d\left( 
\mathbf{x}_{s},\mathbf{x}_{t}\right) $ where $\left\Vert \cdot \right\Vert $
and $d$ are the (Euclidean) Carnot-Caratheodory norm and metric on $%
G^{2}\left( \mathbb{R}^{d}\right) $. In its simplest (non-trivial) setting,
and under mild regularity conditions on the vector fields $V_{1},\dots
,V_{d} $, rough path theory asserts that the ODE solution to 
\begin{equation}
dy=\sum_{i=1}^{d}V_{i}(y)dx^{i},\,\,\,y\left( 0\right) =y_{0}\in \mathbb{R}%
^{e},  \label{ODE}
\end{equation}%
depends continuously (uniformly on bounded sets) on the driving signal $%
x\left( \cdot \right) \in C^{\infty }\left( \left[ 0,1\right] ,\mathbb{R}%
^{d}\right) $ with respect to\footnote{%
It is crucial to have $\alpha >1/3$. The solution map $x\mapsto y$ would not
be continuous in $d_{\alpha \text{-H\"{o}l}}$ for $\alpha \leq 1/3$.} 
\begin{equation*}
\tilde{d}_{\alpha \text{-H\"{o}l}}\left( x,x^{\prime }\right) :=d_{\alpha 
\text{-H\"{o}l}}\left( S_{2}\left( x\right) ,S_{2}\left( x^{\prime }\right)
\right) ,\,\,\,\alpha \in \left( 1/3,1/2\right) ,
\end{equation*}%
where $\mathbf{x}:=S_{2}\left( x\right) $ and 
\begin{equation*}
d_{\alpha \text{-H\"{o}l}}\left( \mathbf{x,x}^{\prime }\right) =\sup_{s,t\in %
\left[ 0,1\right] }\frac{d\left( \mathbf{x}_{s,t},\mathbf{x}_{s,t}^{\prime
}\right) }{\left\vert t-s\right\vert ^{\alpha }}
\end{equation*}%
It is easy to see that\footnote{$\mathbf{1}$ denotes the trivial path
identically equal to the unit element $\mathbf{1\in }$ $G^{2}\left( \mathbb{R%
}^{d}\right) $} 
\begin{equation*}
C^{\alpha \text{-H\"{o}l}}\left( \left[ 0,1\right] ,G^{2}\left( \mathbb{R}%
^{d}\right) \right) :=\left\{ \mathbf{x}:d_{\alpha \text{-H\"{o}l}}\left( 
\mathbf{x,1}\right) <\infty \right\} ,
\end{equation*}%
is a complete metric space for $d_{\alpha \text{-H\"{o}l}}$ (up to constants
:\ $d_{\alpha \text{-H\"{o}l}}\left( \mathbf{x,x}^{\prime }\right) =0$ iff $%
\mathbf{x}_{t}^{-1}\mathbf{x}_{t}^{\prime }\equiv $ constant). It follows
that the very meaning of the ODE (\ref{ODE}) can be extended, in a unique
and continuous fashion, to the $d_{\alpha \text{-H\"{o}l}}$-closure of
lifted smooth paths in $C^{\alpha \text{-H\"{o}l}}\left( \left[ 0,1\right]
,G^{2}\left( \mathbb{R}^{d}\right) \right) .$ This closure is a Polish space
for the metric $d_{\alpha \text{-H\"{o}l}}$ and is denoted by 
\begin{equation*}
C^{0,\alpha \text{-H\"{o}l}}\left( \left[ 0,1\right] ,G^{2}\left( \mathbb{R}%
^{d}\right) \right) .
\end{equation*}%
If $\alpha \in \left( 1/3,1/2\right) $ it is known as the space of \textit{%
geometric }$\alpha $\textit{-H\"{o}lder rough paths\footnote{%
The same construction applies when $G^{2}\left( \mathbb{R}^{d}\right) $ is
replaced by $G^{N}\left( \mathbb{R}^{d}\right) $ in which case one requires $%
\alpha \in \left( 1/\left( N+1\right) ,1/N\right) $ in order to speak of 
\textit{geometric }$\alpha $\textit{-H\"{o}lder rough paths}.}}. It includes
almost every realization of \textit{enhanced Brownian motion} $\mathbf{B}$, 
\begin{equation*}
\mathbf{B}\equiv \left( 1,B,\int B\otimes \circ dB\right)
\end{equation*}%
where $\circ dB$ denotes the Stratonovich differential of $B$. The resulting
\textquotedblleft generalized\textquotedblright\ ODE solution driven by $%
\mathbf{B}$ can then be identified as the classical Stratonovich SDE
solution, \cite{lyons-98, lyons-qian-02, MR2314753}. This provides an
essentially deterministic approach to SDE theory with numerous benefits when
it comes to regularity questions of the It\^{o} map, construction of
stochastic flows, etc. Another property of such rough paths is that there is
a unique (again: modulo constants) lift of $\mathbf{x}\in C^{0,\alpha \text{%
-H\"{o}l}}\left( \left[ 0,1\right] ,G^{2}\left( \mathbb{R}^{d}\right)
\right) $ to a path of similar regularity in the step-$N$ group for all $%
N\geq 2,$%
\begin{equation*}
S_{N}\left( \mathbf{x}\right) \in C^{0,\alpha \text{-H\"{o}l}}\left( \left[
0,1\right] ,G^{N}\left( \mathbb{R}^{d}\right) \right) .
\end{equation*}%
(See \cite[Thm 3.7.]{MR2314753} for instance.) In the case of the enhanced
Brownian motion $\mathbf{x=B}\left( \omega \right) $, the process $%
S_{N}\left( \mathbf{B}\right) $ identifies as Brownian motion $B$ plus all
iterated Stratonovich integrals up to order $N$. In particular, $S_{N}\left( 
\mathbf{B}\right) $ is then realization of Brownian motion on the step-$N$
free nilpotent group.

\bigskip

Our main result is

\begin{theorem}
\label{ThmMainIntro}Assume $\mathbb{E}\left( \left\vert \xi \right\vert
^{2p}\right) <\infty $ for some real number $p\geq 4$. Then 
\begin{equation*}
S_{2}(W_{t}^{\left( n\right) })\implies \mathbf{B}\text{ in }C^{0,\alpha 
\text{-H\"{o}l}}\left( \left[ 0,1\right] ,G^{2}\left( \mathbb{R}^{d}\right)
\right)
\end{equation*}%
provided 
\begin{equation}
\alpha \in \left( \frac{1}{3},\frac{p^{\ast }-1}{2p^{\ast }}\right) \text{
with}\,\,\,p^{\ast }=\min \left( [p],2[p/2]\right)  \label{alphaRange}
\end{equation}%
In particular, if $\mathbb{E}\left( \left\vert \xi \right\vert ^{2p}\right)
<\infty $ for all $p<\infty $ then the weak convergence holds for any $%
\alpha <1/2$.
\end{theorem}

Let us point out that Theorem \ref{ThmMainIntro} implies, of course, $\alpha 
$-H\"{o}lder convergence for \textit{all} $\alpha <\frac{p^{*}-1}{2p^{*}}$
but only for $\alpha >1/3$ do we actually get the interesting corollaries
regarding convergence in rough path topology. Let us also remark that weak
convergence of $S_{N}(W_{t}^{\left( n\right) })$ for all $N$ is a
(deterministic) consequence of Theorem \ref{ThmMainIntro}, cf. Corollary \ref%
{Cor3} below.

Although reminiscent of Lamperti's sharp upper bound on the H\"{o}lder
exponent, $\frac{p-1}{2p}$, the actual \textquotedblleft
coarsened\textquotedblright\ form of our upper bound in (\ref{alphaRange}),
with $p$ replaced by $p^{\ast },$ i.e. the largest even integer smaller or
equal to $p,$ comes from our argument (which requires us to work with
integer powers). To handle the case of \textquotedblleft integrability
level\textquotedblright\ $p\in (1,4)$ it is clear that the step-$2$ setting
will not be sufficient. Indeed, from Lamperti's bound, we would have to work
at least in the step-$N$ group with $N\sim 2p/\left( p-1\right) $. More
precisely, we would need to be able to work in $\alpha $-H\"{o}lder topology
for $G^{N}\left( \mathbb{R}^{d}\right) $-valued paths with 
\begin{equation*}
1/\left( N+1\right) <\alpha <\left( p-1\right) /2p.
\end{equation*}%
Any result of the form 
\begin{equation*}
S_{N}(W_{t}^{\left( n\right) })\implies S_{N}\left( \mathbf{B}\right) \text{
in }C^{0,\alpha \text{-H\"{o}l}}\left( \left[ 0,1\right] ,G^{N}\left( 
\mathbb{R}^{d}\right) \right)
\end{equation*}%
would then be equally interesting as it would constitute a \textquotedblleft
step-$N$\textquotedblright\ convergence result in rough path topology with
similar corollaries as those described below. Unfortunately, as explained in
Section \ref{SecFiniteMoments}, the \textquotedblleft
coarsened\textquotedblright\ H\"{o}lder exponents that we obtain in the step-%
$N$ setting are not bigger than $1/\left( N+1\right) $ in general. Although
we suspect this to be an artefact of our proof, we currently do not know how
to bypass this difficulty in order to handle $p\in (1,4)$.

Let us now discuss some applications to Theorem \ref{ThmMainIntro}. The
afore-mentioned continuity results of rough path theory lead immediately to
corollaries of the following type (recall that we assume existence of a
moment of order $p\geq 4$ for the $\mathbb{\xi }_{i}$'s$).$

\begin{corollary}[Donsker-Wong-Zakai type convergence]
Assume $V=\left( V_{1},\dots ,V_{d}\right) $ is a collection of $C^{3}$%
-bounded vector fields on $\mathbb{R}^{e}$ and let $\left( Y^{n}\right) $
denote the family of (random) ODE solutions to 
\begin{equation*}
dY^{n}=V\left( Y^{n}\right) dW^{\left( n\right) },\,\,\,Y_{0}=y_{0}\in 
\mathbb{R}^{e}.
\end{equation*}
Then, with $\alpha $ as in (\ref{alphaRange}), 
\begin{equation*}
Y^{n}\implies Y\text{ in }C^{\alpha \text{-H\"{o}l}}\left( \left[ 0,1\right]
,\mathbb{R}^{e}\right)
\end{equation*}
where $Y$ is the (up to indistinguishability) unique continuous solution to
the Stratonovich SDE 
\begin{equation*}
dY=V\left( Y\right) \circ dB,\,\,\,Y_{0}=y_{0}\in \mathbb{R}^{e}.
\end{equation*}
\end{corollary}

\begin{remark}
The regularity assumptions of the $\left( V_{i}\right) $ can be slightly
weakened. One can also add a drift vector field (only assumed $C^{1}$%
-bounded say) and in fact the weak convergence can be seen to hold in sense
of flows of $C^{1}$-diffeomorphisms (and then $C^{k}$-flows for $k\in 
\mathbb{N}$, provided additional smoothness assumptions are made on $V$).
Indeed, all this follows from the appropriate (deterministic) continuity
results of rough path theory, cf. for instance, combined with stability of
weak convergence under continuous maps.
\end{remark}

\begin{corollary}[Weak convergence to stochastic integrals]
Assume $\varphi =\left( \varphi _{1},\dots ,\varphi _{d}\right) $ is a
collection of $C_{b}^{2}$-bounded functions from $\mathbb{R}^{d}$ to $%
\mathbb{R}^{e}$. Then 
\begin{equation*}
\int_{0}^{\cdot }\varphi \left( W^{\left( n\right) }\right) dW^{\left(
n\right) }\implies \int_{0}^{\cdot }\varphi \left( B\right) \circ dB.
\end{equation*}
\end{corollary}

\begin{corollary}[Convergence to BM\ on the free step-$N$ nilpotent group]
\label{Cor3}Assume $N\geq 2$. Then 
\begin{equation*}
S_{N}(W_{t}^{\left( n\right) })\implies \mathbf{\tilde{B}}\text{ in }%
C^{0,\alpha \text{-H\"{o}l}}\left( \left[ 0,1\right] ,G^{N}\left( \mathbb{R}%
^{d}\right) \right)
\end{equation*}%
with $\alpha $ as in (\ref{alphaRange}) and $\mathbf{\tilde{B}}$ a Brownian
motion on $G^{N}\left( \mathbb{R}^{d}\right) $.
\end{corollary}

\begin{acknowledgement}
The second author is partially supported by a Leverhulme Research Fellowship
and EPSRC grant EP/E048609/1.
\end{acknowledgement}

\section{Donsker's theorem for enhanced Brownian motion and random walks on
groups}

We first discuss the case of a random walk for finite moments of all orders.

\begin{theorem}[Donsker's theorem for enhanced Brownian motion]
\bigskip Assume $\mathbb{E\xi }=0$ and $\mathbb{E}\left( \left| \xi \right|
^{p}\right) <\infty $ for all $p\in [1,\infty )$ and $\alpha <1/2$. Then 
\begin{equation*}
S_{2}(W_{\cdot }^{\left( n\right) })\implies \mathbf{B}\text{ in }%
C^{0,\alpha \text{-H\"{o}l}}\left( \left[ 0,1\right] ,G^{2}\left( \mathbb{R}%
^{d}\right) \right)
\end{equation*}
where $\mathbf{B}$ is a ($G^{2}\left( \mathbb{R}^{d}\right) $-valued)
enhanced Brownian motion.
\end{theorem}

In fact, we shall prove a more general theorem that deals with random walks
on groups. More precisely, by a theorem of Chen \cite{chen-57} we have 
\begin{equation*}
S_{2}\left( W^{\left( n\right) }\right) _{t}=\delta _{n^{-1/2}}\left( e^{\xi
_{1}}\otimes \dots \otimes e^{\xi _{\left[ nt\right] }}\otimes e^{\left( nt-%
\left[ nt\right] \right) \xi _{\left[ nt\right] +1}}\right)
\end{equation*}%
where $\delta $ denotes dilation on $G^{2}\left( \mathbb{R}^{d}\right) $ and 
$e^{v}=\left( 1,v,v^{\otimes 2}/2\right) \in G^{2}\left( \mathbb{R}%
^{d}\right) ,$ $v\in \mathbb{R}^{d}$. Observe that $\left( \mathbf{\xi }%
_{i}\right) :=\left( e^{\xi _{i}}\right) $ is a sequence of independent,
identically distributed $G^{2}\left( \mathbb{R}^{d}\right) $-valued random
variables, \textit{centered }in the sense that 
\begin{equation*}
\mathbb{E}\left( \pi _{1}\left( \mathbf{\xi }_{i}\right) \right) =\mathbb{E}%
\xi _{i}=0,
\end{equation*}%
where $\pi _{1}$ is the projection from $G^{2}\left( \mathbb{R}^{d}\right)
\rightarrow \mathbb{R}^{d}$. Let us also observe that the shortest path
which connects the unit element $1\in $ $G^{2}\left( \mathbb{R}^{d}\right) $
with $e^{\xi _{i}}$ is precisely $e^{t\xi _{i}}$ so that piecewise linear
interpolation on $\mathbb{R}^{d}$ lifts to geodesic interpolation on $%
G^{2}\left( \mathbb{R}^{d}\right) $. We shall thus focus on the following
Donsker-type theorem.

\begin{theorem}
\label{Donskertype}Let $\left( \mathbf{\xi }_{i}\right) $ be a centered IID
sequence of $G^{2}\left( \mathbb{R}^{d}\right) $-valued random variables
with finite moments of all orders, 
\begin{equation*}
\forall q\in [1,\infty ):E\left( \left\| \mathbf{\xi }_{i}\right\|
^{q}\right) <\infty
\end{equation*}
and consider the rescaled random walk defined by $\mathbf{W}_{0}^{\left(
n\right) }=1$ and\bigskip 
\begin{equation*}
\mathbf{W}_{t}^{\left( n\right) }=\delta _{n^{-1/2}}\left( \mathbf{\xi }%
_{1}\otimes \dots \otimes \mathbf{\xi }_{\left[ tn\right] }\right)
\end{equation*}
for $t\in \left\{ 0,\frac{1}{n},\frac{2}{n},\dots \right\} $,
piecewise-geodesically-connected in between (i.e. $\mathbf{W}_{t}^{\left(
n\right) }|_{\left[ \frac{i}{n},\frac{i+1}{n}\right] }$ is a geodesic
connecting $\mathbf{W}_{i/n}^{\left( n\right) }$ and $\mathbf{W}_{\left(
i+1\right) /n}^{\left( n\right) }$). Then, for any $\alpha <1/2$, $\mathbf{W}%
^{\left( n\right) }$ converges weakly to $\mathbf{B}$, in $C^{0,\alpha \text{%
-H\"{o}l}}\left( \left[ 0,1\right] ,G^{2}\left( \mathbb{R}^{d}\right)
\right) .$
\end{theorem}

\section{Proof of theorem \protect\ref{Donskertype}}

Following a standard pattern of proof, weak convergence follows from
convergence of the finite-dimensional-distributions and tightness (here in $%
\alpha $-H\"{o}lder topology).

\underline{Step 1: } (Convergence of the finite-dimensional-distributions)
This is an immediate consequence of the :

\begin{theorem}
(\textit{Central limit theorem for centered i.i.d variables on a nilpotent
Lie group}) Let $N$ be a simply connected nilpotent Lie group and $\mathbf{%
\xi }_{1},...,\mathbf{\xi }_{n},...$ be i.i.d. random variables with values
in $N$ which we assume centered (i.e. their projection $\pi _{1}(\mathbf{\xi 
}_{i})$ on the abelianization of $N$ has zero mean) and with a finite second
moment, i.e. $E\left( \left\| \mathbf{\xi }_{i}\right\| ^{2}\right) <\infty
. $ Then we have the following convergence in law: 
\begin{equation*}
\delta _{n^{-1/2}}\left( \mathbf{\xi }_{1}\otimes \dots \otimes \mathbf{\xi }%
_{n}\right) \rightarrow \mathbf{B}_{1}
\end{equation*}
where $\mathbf{B}_{1}$ is the time $1$ value of the Brownian motion on $N$
associated to $\mathbf{\xi }_{1}$ (i.e. the symmetric diffusion on $N$ with
infinitesimal generator the left invariant sub-Laplacian $\frac{1}{2}\sum
a_{ij}X_{i}X_{j}$ where $(a_{ij})$ is the covariance matrix of $\pi _{1}(\xi
_{i})$).
\end{theorem}

This theorem is a straightforward consequence of the main result of Wehn's
(unpublished) 1962 thesis, cf. \cite{MR0206994}, \cite[Thm 3.11]%
{breuillard-07} or \cite{MR507730}. It also follows a fortiori from the much
stronger Stroock-Varadhan Donsker-type theorem in connected Lie groups \cite%
{MR0517406}.

\underline{Step 2:} (Tightness) We need to find positive constants $a,b,c$
such that for all $u,v\in \left[ 0,1\right] $%
\begin{equation*}
\sup_{n}\mathbb{E}\left[ d\left( \mathbf{W}_{v}^{\left( n\right) },\mathbf{W}%
_{u}^{\left( n\right) }\right) ^{a}\right] \leq c\left\vert v-u\right\vert
^{1+b},
\end{equation*}%
then we can apply Kolmogorov's tightness criterion\footnote{%
E.g. \cite{revuz-yor-1999}; the extension from the real-valued processes to $%
\left( G^{2}\left( \mathbb{R}^{d}\right) ,d\right) $-valued processes, with $%
d$ being the Carnot-Caratheodory metric, is trivial.} to obtain tightness
in\ $\gamma $-H\"{o}lder topology, for any $\gamma <b/a$. Using basic
properties of geodesic interpolation, we see that it is enough to consider $%
u,v\in \left\{ 0,\frac{1}{n},\frac{2}{n},\dots \right\} $ and then (by
independence of increments and left invariance of $d$) there is no loss of
generality in taking $\left[ u,v\right] =\left[ 0,k/n\right] $ for some $%
k\in \left\{ 0,\dots ,n\right\} $. It follows that what has to be
established reads 
\begin{equation*}
\exists a,b,c_{1}:\frac{1}{n^{a/2}}E\left[ \left\Vert \mathbf{\xi }%
_{1}\otimes \dots \otimes \mathbf{\xi }_{k}\right\Vert ^{a}\right] \leq
c_{1}\left\vert \frac{k}{n}\right\vert ^{1+b},
\end{equation*}%
uniformly over all $n\in \mathbb{N}$ and $0\leq k\leq n$, and such that $b/a$
can be taken arbitrarily close to $1/2$. To this end, it is enough to show
that for all $p\in \left\{ 1,2,\dots \right\} $%
\begin{equation*}
\left( \ast \right) :\mathbb{E}\left[ \left\Vert \mathbf{\xi }_{1}\otimes
\dots \otimes \mathbf{\xi }_{k}\right\Vert ^{4p}\right] =O\left(
k^{2p}\right)
\end{equation*}%
since we can then take $a=4p,b=2p-1$ and of course $b/a=\left( 2p-1\right)
/\left( 4p\right) \uparrow 1/2$ as $p\uparrow \infty $. Thus, the proof is
finished once we show $\left( \ast \right) $ and this is the content of the
last step of this proof.

\underline{Step 3:} Let $P$ be a polynomial function on $G^{2}(\mathbb{R}%
^{d}),$ i.e. a polynomial in $a^{1;i},a^{2;ij}$ where 
\begin{equation*}
a=\left( a^{1;i},a^{2;ij};1\leq i\leq d,1\leq i<j\leq d\right) \in \mathfrak{%
g}^{2}\left( \mathbb{R}^{d}\right)
\end{equation*}%
is the log-chart of $G^{2}(\mathbb{R}^{d})$, $g\mapsto a=\log \left(
g\right) $. We define the \textit{degree} $d^{\circ }P$ by agreeing that
monomials of form 
\begin{equation*}
\left( a^{1;i}\right) ^{\alpha _{i}}\left( a^{2;ij}\right) ^{\alpha _{i,j}}
\end{equation*}%
have degree $\sum \alpha _{i}+2\sum \alpha _{i,j}$. If $\mathbf{\xi }$ is a $%
G^{2}(\mathbb{R}^{d})$-valued random variable with moments of all orders,
then 
\begin{equation*}
TP:g\mapsto \mathbb{E}\left( P\left( g\otimes \mathbf{\xi }\right) \right)
-P\left( g\right)
\end{equation*}%
is well defined and is another polynomial function. Moreover if $\mathbf{\xi 
}$ is centered, an easy application of the Campbell-Baker-Hausdorff formula
reveals that $TP$ is of degree $\leq d^{o}P-2$. (For instance, $P\left(
a\right) :=\left( a^{2;ij}\right) ^{m}$ has degree $2m$; then $TP$ is seen
to contain terms of the form $\left( a^{2;ij}\right) ^{m-1}$ and $\left(
a^{2;ij}\right) ^{m-2}\left( a^{1;k}\right) ^{2}$ etc. all of which are
indeed of degree $2m-2$). Now, for any $p\in \left\{ 1,2,\dots \right\} $, 
\begin{eqnarray*}
\left\Vert e^{a}\right\Vert ^{4p} &\sim &\sum_{i}\left\vert
a^{1;i}\right\vert ^{4p}+\sum_{i<j}\left\vert a^{1;ij}\right\vert ^{2p} \\
&=&\sum_{i}\left( a^{1;i}\right) ^{4p}+\sum_{i<j}\left( a^{1;ij}\right)
^{2p}=:P\left( e^{a}\right)
\end{eqnarray*}%
where $P$ is a polynomial of degree $4p$. Recalling the definition of the
operator $T$ and using independence we have, 
\begin{eqnarray*}
\mathbb{E}\left[ P(\mathbf{\xi }_{1}\otimes \dots \otimes \mathbf{\xi }_{k})%
\right] &=&\mathbb{E}\left[ \mathbb{E}\left[ P((\mathbf{\xi }_{1}\otimes
\dots \otimes \mathbf{\xi }_{k-1})\otimes \mathbf{\xi }_{k})\mid \mathbf{\xi 
}_{1},...,\mathbf{\xi }_{k-1}\right] \right] \\
&=&\mathbb{E}\left[ TP(\mathbf{\xi }_{1}\otimes \dots \otimes \mathbf{\xi }%
_{k-1})+P(\mathbf{\xi }_{1}\otimes \dots \otimes \mathbf{\xi }_{k-1})\right]
\\
&=&...= \\
&=&(T+1)^{k}P(1) \\
&=&\sum_{l\geq 0}\binom{k}{l}T^{l}P(1)
\end{eqnarray*}%
But the function $TP:g\mapsto \mathbb{E}(P(g\otimes \mathbf{\xi }))-P(g)$ is
a polynomial function of degree at most $d^{\circ }P-2=4p-2$. Hence 
\begin{equation*}
d^{\circ }T^{l}P\leq d^{\circ }P-2l=2\left( 2p-l\right)
\end{equation*}%
and the above sum contains only a finite number of terms, more precisely 
\begin{equation*}
\mathbb{E}\left[ P(\mathbf{\xi }_{1}\otimes \dots \otimes \mathbf{\xi }_{k})%
\right] =\sum_{l=0}^{2p}\binom{k}{l}T^{l}P(1).
\end{equation*}%
Since each of these terms is $O(k^{2p})$, as $k\rightarrow \infty $, we are
done.

\section{Extension to finite moments\label{SecFiniteMoments}}

\bigskip The question remains what happens if we weaken the moment
assumption to 
\begin{equation*}
\mathbb{E}\left( \left\Vert \mathbf{\xi }_{i}\right\Vert ^{2p}\right)
<\infty \text{ for some}\,\,\,\,p>1.
\end{equation*}%
where, for now, $\mathbf{\xi }_{i}\in G^{2}\left( \mathbb{R}^{d}\right) $.
If we could prove that 
\begin{equation}
\mathbb{E}\left[ \left\Vert \mathbf{\xi }_{1}\otimes \dots \otimes \mathbf{%
\xi }_{k}\right\Vert ^{2p}\right] =O\left( k^{p}\right)
\label{WishfulThinking}
\end{equation}%
then the arguments of the previous section apply line-by-line to obtain
tightness (and hence weak convergence)\ in $C^{0,\gamma \text{-H\"{o}l}%
}\left( \left[ 0,1\right] ,G^{2}\left( \mathbb{R}^{d}\right) \right) $ , for
any 
\begin{equation*}
\gamma <\frac{p-1}{2p}.
\end{equation*}%
In the case that $\frac{p-1}{2p}>1/3$ we can and will choose $\gamma \in
\left( 1/3,1/2\right) $ since then $C^{0,\gamma \text{-H\"{o}l}}\left( \left[
0,1\right] ,G^{2}\left( \mathbb{R}^{d}\right) \right) $ is a genuine rough
path space\footnote{%
The integer part of $1/\gamma $ must match the level of nilpotency: $\left[
1/\gamma \right] =2$.}. Otherwise, i.e. if $\gamma <\frac{p-1}{2p}\leq 1/3$,
tightness in $C^{0,\gamma \text{-H\"{o}l}}\left( \left[ 0,1\right]
,G^{2}\left( \mathbb{R}^{d}\right) \right) $ is worthless (from the point of
view of rough path applications). However, we can still ask for the smallest
integer $N$ such that 
\begin{equation*}
\frac{p-1}{2p}>\frac{1}{N+1}
\end{equation*}%
and consider $G^{N}\left( \mathbb{R}^{d}\right) $-valued random variables $%
\mathbf{\xi }_{i}$ with finite $\left( 2p\right) $-moments. Again, if we
could show that (\ref{WishfulThinking}) holds, all arguments extend and we
would obtain tightness in $C^{0,\gamma \text{-H\"{o}l}}\left( \left[ 0,1%
\right] ,G^{N}\left( \mathbb{R}^{d}\right) \right) $ where $\gamma $ can be
chosen to be in $\left( \frac{1}{N+1},\frac{1}{N}\right) $ so that we have
tightness in a \textquotedblright step-$N$ rough path
topology\textquotedblright . Unfortunately, the \textquotedblright
polynomial\textquotedblright\ proof of the previous section does not allow
to obtain (\ref{WishfulThinking}) but only the following slightly weaker
result.

\begin{proposition}
Let $\left( \mathbf{\xi }_{i}\right) $ be a centered IID sequence of $%
G^{N}\left( \mathbb{R}^{d}\right) $-valued random variables with finite $%
\left( 2p\right) $-moments, $p$ real $\,p>1$%
\begin{equation}
E\left( \left\| \mathbf{\xi }_{i}\right\| ^{2p}\right) <\infty .  \label{mom}
\end{equation}
Then 
\begin{equation}
\mathbb{E}\left[ \left\| \mathbf{\xi }_{1}\otimes \dots \otimes \mathbf{\xi }%
_{k}\right\| ^{2q}\right] =O\left( k^{q}\right)  \label{PropOkp}
\end{equation}
for all $q\leq q_{0}\left( p,N\right) \leq p$ where $q_{0}\left( p,N\right) $
may be taken to be 
\begin{equation*}
q_{0}\left( p,N\right) =\min_{m=1,\dots ,N}m\left[ \frac{p}{m}\right] .
\end{equation*}
\end{proposition}

\begin{proof}
Set $q_{m}=m\left[ p/m\right] $. The conclusion is equivalent to%
\begin{equation*}
\left\vert \left\Vert \mathbf{\xi }_{1}\otimes \dots \otimes \mathbf{\xi }%
_{k}\right\Vert \right\vert _{L^{2q}}=O\left( k^{1/2}\right)
\end{equation*}%
and will follow, with $q=\min \left\{ q_{m}:m=1,\dots ,N\right\} $ from%
\begin{eqnarray*}
&&\left\vert \left\vert \pi _{m}\left( \log \left( \mathbf{\xi }_{1}\otimes
\dots \otimes \mathbf{\xi }_{k}\right) \right) \right\vert ^{1/m}\right\vert
_{L^{2q}} \\
&\leq &\left\vert \left\vert \pi _{m}\left( \log \left( \mathbf{\xi }%
_{1}\otimes \dots \otimes \mathbf{\xi }_{k}\right) \right) \right\vert
^{1/m}\right\vert _{L^{2q_{m}}}=O\left( k^{1/2}\right)
\end{eqnarray*}%
provided we can show the $O\left( k^{1/2}\right) $-estimate of the last
line, for all $m\in \left\{ 1,\dots ,N\right\} $. To this end, consider $%
P_{m}\left( e^{a}\right) $ given by%
\begin{equation*}
\sum_{i_{1}\dots i_{m}\in \left\{ 1,\dots ,d\right\} }\left( a^{m;i_{1}\dots
i_{m}}\right) ^{2[p/m]}
\end{equation*}%
and observe that it has homogenous degree $2m\left[ p/m\right] \leq 2p$. Now
observe that the condition of existence of moments of order $2p$ for $%
\mathbf{\xi }$, i.e. (\ref{mom}), implies (using H\"{o}lder's inequality)
that the expectation of any polynomial in $\mathbf{\xi }$ of homogeneous
degree at most $2p$ is finite. Hence the operator $T$ defined in Step 3 of
the proof of Theorem \ref{Donskertype} applies to $P_{m}$ and the argument
there (see also \cite[Lemma 2.4]{MR507730} shows that it will reduce its
degree by two. The same arguments as earlier will then give us the $O\left(
k^{1/2}\right) $-estimate with $q_{m}$ determined by%
\begin{equation*}
2q_{m}/m=2\left[ p/m\right] \text{ or }q_{m}=m\left[ p/m\right] .
\end{equation*}
\end{proof}

\bigskip

As in the previous section, (\ref{PropOkp}) with $q=q_{0}$ implies $\alpha $%
-H\"{o}lder tightness/convergence for any $\alpha <\left( q_{0}-1\right)
/2q_{0}$. In particular, when 
\begin{equation*}
\frac{q_{0}-1}{2q_{0}}>\frac{1}{N+1}\text{ or~}q_{0}=\min_{m=1,\dots ,N}m%
\left[ \frac{p}{m}\right] >\frac{N+1}{N-1}=1+\frac{2}{N-1}
\end{equation*}
we can pick $\alpha >\frac{1}{N+1}$ and thus have ''rough path convergence''.%
\newline
\underline{Case $N=2$:} In this case, the above condition reduces to $\min
([p],2[p/2])>3$ which requires $p\geq 4$. Indeed, for $p=4$ we have $%
q_{0}\left( 4,N\right) =4$ which implies $\alpha $-H\"{o}lder
tightness/convergence for any 
\begin{equation*}
\alpha <\frac{q_{0}-1}{2q_{0}}=\frac{3}{8}
\end{equation*}
and since $3/8\in \left( 1/3,1/2\right) $ we can indeed pick $\alpha \in
\left( 1/3,1/2\right) $, such as to have tightness/convergence in rough path
topology. Remark that any $p<4$ implies $q_{0}\left( p,N\right) \leq 3$ so
that we necessarily have to pick 
\begin{equation*}
\alpha <\frac{3-1}{6}=\frac{1}{3}
\end{equation*}
in which case $\alpha $-H\"{o}lder topology is not a rough path topology on
the space of $G^{N}\left( \mathbb{R}^{d}\right) =G^{2}\left( \mathbb{R}%
^{d}\right) $-valued paths.\newline
\underline{Case $N=3$:} In this case, the above condition reduces to $\min
([p],2[p/2],3\left[ p/3\right] )>2$ which also requires $p\geq 4$. But by
assuming $p\geq 4$ we can safely work in the step-$2$ group. \newline
\underline{Case $N\geq 4$:} In this case, $\min ([p],2[p/2],3\left[ p/3%
\right] ,4\left[ p/4\right] ,\dots )>\left[ 1+\frac{2}{N-1}\right] =1$ so
that 
\begin{equation*}
m\left[ p/m\right] >1\text{ and hence }\geq 2
\end{equation*}
for all $m\in \left\{ 1,\dots ,N\right\} $. In particular we can take $m=4$
and see $4\left[ p/4\right] \geq 2\implies $ $\left[ p/4\right] \geq
1\implies p\geq 4$. Then again, we can make the remark that under the
assumption $p\geq 4$ we can safely work in the step-$2$ group.\newline

\begin{corollary}
Assume $\mathbb{E}\left( \left\vert \xi \right\vert ^{2p}\right) <\infty $
for some real $p\geq 4$. Then the rescaled (step-2) lift of the (rescaled,
piecewise linearly connected) random walk $W_{t}^{\left( n\right) }$
converges in $\alpha $-H\"{o}lder for any 
\begin{equation*}
\frac{1}{3}<\alpha <\frac{p^{\ast }-1}{2p^{\ast }}\equiv \alpha ^{\ast
}\left( p\right)
\end{equation*}%
where $p^{\ast }=\min ([p],2[p/2]).$
\end{corollary}

\begin{remark}
For $p\in \left\{ 4,6,8,\dots \right\} $, $\alpha ^{*}\left( p\right) =\frac{%
p-1}{2p}$. In particular, $\alpha ^{*}\left( 4\right) =3/8$ and 
\begin{equation*}
\alpha ^{*}\left( p\right) \sim \frac{p-1}{2p}\rightarrow 1/2\text{ as }%
p\rightarrow \infty
\end{equation*}
in agreement with Theorem \ref{Donskertype}.
\end{remark}

\bibliographystyle{plain}
\bibliography{roughpaths}

\def\cprime{$'$} \def\cprime{$'$}
\begin{thebibliography}{10}

\bibitem{breuillard-07}
E.~Breuillard.
\newblock Random walks on lie groups.
\newblock preprint, 2007.

\bibitem{chen-57}
Kuo-Tsai Chen.
\newblock Integration of paths, geometric invariants and a generalized
  {B}aker-{H}ausdorff formula.
\newblock {\em Ann. of Math. (2)}, 65:163--178, 1957.

\bibitem{MR507730}
Pierre Cr{\'e}pel and Albert Raugi.
\newblock Th\'eor\`eme central limite sur les groupes nilpotents.
\newblock {\em Ann. Inst. H. Poincar\'e Sect. B (N.S.)}, 14(2):145--164, 1978.

\bibitem{MR0206994}
Ulf Grenander.
\newblock {\em Probabilities on algebraic structures}.
\newblock John Wiley \& Sons Inc., New York, 1963.

\bibitem{MR0143245}
John Lamperti.
\newblock On convergence of stochastic processes.
\newblock {\em Trans. Amer. Math. Soc.}, 104:430--435, 1962.

\bibitem{lyons-98}
Terry Lyons.
\newblock Differential equations driven by rough signals.
\newblock {\em Rev. Mat. Iberoamericana}, 14(2):215--310, 1998.

\bibitem{lyons-qian-02}
Terry Lyons and Zhongmin Qian.
\newblock {\em System {C}ontrol and {R}ough {P}aths}.
\newblock Oxford University Press, 2002.
\newblock Oxford Mathematical Monographs.

\bibitem{MR2314753}
Terry~J. Lyons, Michael Caruana, and Thierry L{\'e}vy.
\newblock {\em Differential equations driven by rough paths}, volume 1908 of
  {\em Lecture Notes in Mathematics}.
\newblock Springer, Berlin, 2007.
\newblock Lectures from the 34th Summer School on Probability Theory held in
  Saint-Flour, July 6--24, 2004, With an introduction concerning the Summer
  School by Jean Picard.

\bibitem{revuz-yor-1999}
Daniel Revuz and Marc Yor.
\newblock {\em Continuous martingales and {B}rownian motion}, volume 293 of
  {\em Grundlehren der Mathematischen Wissenschaften [Fundamental Principles of
  Mathematical Sciences]}.
\newblock Springer-Verlag, Berlin, third edition, 1999.

\bibitem{MR0517406}
Daniel~W. Stroock and S.~R.~S. Varadhan.
\newblock Limit theorems for random walks on {L}ie groups.
\newblock {\em Sankhy\=a Ser. A}, 35(3):277--294, 1973.

\bibitem{MR1009453}
Joseph~C. Watkins.
\newblock Donsker's invariance principle for {L}ie groups.
\newblock {\em Ann. Probab.}, 17(3):1220--1242, 1989.

\end{thebibliography}

\end{document}